\documentclass[11pt]{amsart}

\usepackage{amscd}
\usepackage{xypic}  
\usepackage{amssymb}
\usepackage{amsthm}
\usepackage{epsfig}
\usepackage{url}
\usepackage{tikz}

\usepackage[all,cmtip]{xy}
\usepackage{amsmath}
\usepackage{color}
\DeclareMathOperator{\coker}{coker}

\newtheorem{thm}{Theorem}[section]  

\newtheorem{lemma}[thm]{Lemma}

\newtheorem{proposition}[thm]{Proposition}

\newtheorem{corollary}[thm]{Corollary}

\theoremstyle{definition}

\newtheorem{remark}[thm]{Remark}

 \newtheorem{example}[thm]{Example}

\def\ker{\operatorname{ker}}
\def\coker{\operatorname{coker}}

\def\codim{\operatorname{codim}}

\def\im{\operatorname{im}}

\def\c1{\operatorname{c_1}}
\def\c2{\operatorname{c_2}}

\def\Grass{\operatorname{Grass}}

\def\rk{\operatorname{rk}}

\def\CC{{\mathbb C}}

\def\ZZ{{\mathbb Z}}

\def\PP{{\mathbb P}}

\def\DD{{\mathbb D}}

\def\G{{\mathcal G}}
\def\L{{\mathcal L}}
\def\M{{\mathcal M}}
\def\N{{\mathcal N}}
\def\O{{\mathcal O}}
\def\I{{\mathcal J}}

\def\E{{\mathcal E}}

\def\H{{\mathcal H}}
\def\F{{\mathcal F}}
\def\K{{\mathcal K}}
\def\U{{\mathcal U}}
\def\V{{\mathcal V}}

\def\Q{{\mathcal Q}} 
\def\K{{\mathcal K}}

\def\ff{\mathfrak{F}}
\def\c{\mathfrak{c}}

\def\UU{\mathfrak{U}}

\def\x{\times}                   
\def\cong{\simeq}

\def\+{\oplus}               
\def\*{\otimes}                  

\def\id{\operatorname{id}}

\def\Ext{\operatorname{Ext}}
\def\Hom{\operatorname{Hom}}

\def\Shext{\operatorname{ \mathfrak{e}\mathfrak{x}\mathfrak{t} }}

\def\Pic{\operatorname{Pic}}

\def\det{\operatorname{det}}

\def\geq{\geqslant}
\def\leq{\leqslant}

\author[C.~Ciliberto]{Ciro Ciliberto}
\address{Ciro Ciliberto, Dipartimento di Matematica, Universit{\`a} di Roma Tor Vergata, Via della Ricerca Scientifica, 00173 Roma, Italy}
\email{cilibert@mat.uniroma2.it}

\author[F.~Flamini]{Flaminio Flamini}
\address{Flaminio Flamini, Dipartimento di Matematica, Universit{\`a} di Roma Tor Vergata, Via della Ricerca Scientifica, 00173 Roma, Italy} 
\email{flamini@mat.uniroma2.it}

\author[A.~L.~Knutsen]{Andreas Leopold Knutsen}
\address{Andreas Leopold Knutsen, Department of Mathematics, University of Bergen, Postboks 7800,
5020 Bergen, Norway}
\email{andreas.knutsen@math.uib.no}

\title{Ulrich bundles on Del Pezzo threefolds}
\begin{document}

\maketitle

\begin{abstract}  We prove that for any $r \geq 2$ the moduli space of   stable   Ulrich bundles of rank $r$ and determinant $\O_X(r)$ on any smooth Fano threefold $X$ of index two   is smooth    of dimension $r^2+1$   and that the same holds true for even $r$ when the index is four, in which case 
		no odd--rank Ulrich bundles exist.   In particular this shows that any such threefold is {\it Ulrich wild}. 
		As a preliminary result, we give necessary and sufficient conditions for the existence  
		of  Ulrich bundles on any smooth projective threefold in terms of the existence of a curve in the threefold enjoying special properties.
\end{abstract}

\section{Introduction}   A vector bundle $\E$ on a smooth (complex) projective variety $X \subset \PP^{m}$  is said to be an \emph{Ulrich bundle} if $h^i(\E(-p))=0$ for all $i \geq 0$ and all $1 \leq p \leq \dim(X)$.   
The study of  such  bundles  originates in the paper \cite{U}  from  1984, where  they were  considered in the framework of commutative algebra because they enjoy suitable extremal cohomological properties. 
 The  attention of algebraic geometers  was drawn  by  the  paper  \cite{ESW}, where, among other things, the  Chow form  of a projective variety $X$  is computed  using Ulrich bundles on $X$, if they exist.  

In recent years there has been a lot of work  on Ulrich bundles, mainly investigating the following problems: given  a variety $X \subset \PP^{m}$,  does there always exist an Ulrich bundle  on $X$? What  are  the  possible ranks for Ulrich  bundles?  If   Ulrich bundles exist, are they stable, and what are their moduli?   (Recall that 
Ulrich bundles are always semistable and the notions of stability and slope--stability coincide, cf. \cite[Thm.\;2.9]{ch}).  

Although a lot is known about these problems  for specific  classes of varieties (e.g., curves, Segre and Grassmann varieties, rational scrolls, complete intersections, some classes of surfaces and threefolds, etc.) the above questions  are still open in their full  generality (for  nice surveys on the subject, see for instance \cite{be2,ch,Co,CMP}). 

   In this paper we consider smooth Fano threefolds of {\em even index}, also called {\em Del Pezzo threefolds} (cf. the beginning of \S\;\ref{S:bundles} for the definition and classification of such varieties). In the index--two case our main result is:  
 
\begin{thm}\label{thm:Fanoind2}  Let $X \subset \PP^{m}$  be a smooth  Fano threefold of index two with \linebreak $\omega_X \cong \O_X(-2)$.   For every  integer $r \geqslant 2$, the moduli space  
  of    stable   Ulrich bundles of rank $r$ and determinant $\O_X(r)$ on $X$   is smooth   of dimension  $r^2+1$.  
\end{thm}

  In the remaining case of even index we prove: 
	
	\begin{thm} \label{thm:veronese}  Let $X \subset \PP^{9}$  be the $2$--Veronese embedding of $\PP^3$.    Then: 
	
	\begin{itemize}
	\item[(i)] $X$ carries no Ulrich bundles of odd rank; 
	
	\item[(ii)] for every  even integer $r \geqslant 2$, the moduli space  
  of   stable  Ulrich bundles of rank $r$ and determinant $\O_X(r)$ on $X$  is smooth of dimension  $r^2+1$.  
	\end{itemize}
\end{thm}

 Our results in particular show that any such threefold $X$  is \emph{Ulrich wild} (recall that, as suggested by an analogous definition in \cite {DG}, a variety $X$ is said to be Ulrich wild if it possesses families of dimension $p$ of pairwise non--isomorphic, indecomposable, Ulrich bundles for arbitrarily large $p$). Note  that  there  are only very few cases of varieties known to carry stable Ulrich bundles of infinitely many ranks, and even fewer of \emph{any} rank, namely curves, Del Pezzo surfaces and more recently  general blow-ups of the plane  (cf.\,\cite{CFK}). 

 Some special cases of the theorems   were known already: the rank--two cases were proved by Beauville in \cite[Prop. 6.1 and 6.4]{be2}   and, in the case of the $2$--Veronese, already in \cite[Prop.\;5.10--5.11]{ESW},   without the statement about stability. The case of   arbitrary rank on   the cubic threefold in $\PP^4$ was proved in \cite[Prop. 5.4 and Thm. 5.7]{ch} under the additional hypothesis that the threefold is {\it general}, and for {\it any} cubic threefold in \cite[Thm. B]{LMS}\color{black}; the case of intersections of two quadrics ($d=4$) was proved in \cite[Thm. 1.1]{CKL2}, and the case $d=5$ was proved in  \cite[Thm. 1.1]{CKL1}.   \color{black} Our proof   of Theorem \ref{thm:Fanoind2}    is different from \color{black} the ones in these papers, and provides a uniform treatment of all Del Pezzo threefolds, \color{black}   whereas our proof of Theorem \ref{thm:veronese}(ii) is similar to 
 the one in \cite{ch}. \color{black} Theorem \ref {thm:Fanoind2} proves the conjecture stated in \cite[p. 276]{CKL1}.\color{black}

Also note that the cases of rank--one Ulrich bundles  on     Del Pezzo   threefolds are well--known (cf., e.g., \cite{be2,CFaM1,CFaM2,CFiM}  and see  the proof of Proposition \ref{prop:beauville-stable}). For other works regarding Ulrich bundles, or more generally $ACM$ bundles, of rank two on Fano threefolds, we refer to \cite{ac,be0,be1,be3,bf,CFaM1,CFaM2,CFiM,mt}.

We do not claim the irreducibility of the moduli spaces    in Theorems \ref{thm:Fanoind2} and \ref{thm:veronese}.    
Something is known in the rank--two case, cf. Remark \ref{rem:duecomp}, and
\color{black} irreducibility for all ranks in the cases $d=4$ and $5$ follows from the explicit description of the moduli spaces given in \cite[Thm. 1.1]{CKL2} and \cite[Thm. 1.1]{CKL1}, whereas it was 
very recently \color{black} proved in the case of the cubic threefold in \cite[Thm. 1.4]{FP}. Irreducibility in the remaining cases is an interesting open question.   

  Theorems \ref{thm:Fanoind2} and \ref{thm:veronese}(ii) will be proved in \S\;\ref{S:bundles}.    
Before that,  in Theorem\;\ref{thm:codim2}, we  give  necessary and sufficient conditions for 
the existence of a rank $r \geqslant 2$ Ulrich bundle on any smooth projective threefold  $X \subset \PP^m$   in terms of  the existence of  a quite peculiar curve $C$ 
on $X$,  and the peculiarity of this curve may explain  why Ulrich bundles often  seem  quite hard to find.
 Despite  the complexity of the conditions that the curve  has to verify, in  certain 
cases they simplify a bit,  as in the case of Fano threefolds: see Theorem \ref{thm:fano} and Corollaries \ref{cor1} and \ref{cor2},   the latter proving
Theorem \ref{thm:veronese}(i).   We finish the paper with some remarks and speculations concerning  interesting maps to moduli spaces of curves arising from  this correspondence  between Ulrich bundles and curves  (cf. \S\,\ref{S:moduli}).

 Throughout the paper we work over the field of complex numbers. 

\medskip

{\bf Acknowledgements:} C.~Ciliberto and F.~ Flamini are members of GNSAGA of the Istituto Nazionale di Alta Matematica ``F. Severi"  and acknowledge support from  the MIUR Excellence Department Project awarded to the Department of Mathematics, University of Rome Tor Vergata, 
CUP E83C18000100006.  A.~L.~Knutsen acknowledges support from the Trond Mohn Foundation Project ``Pure Mathematics in Norway'' and   the Meltzer Research Fund   grant 261756 of the
Research Council of Norway. The authors thank G.~Casnati for interesting comments concerning Remark \ref{rem:duecomp},  Kyoung-Seog Lee for interesting correspondence and \color{black} the Referee for useful feedback. \color{black}

\section{ A preliminary lemma }\label{S:prel} 

 This short section is devoted to the proof of the following:

\begin{lemma} \label{lemma:dualnosez} Let   $X \subset \PP^m$  be a smooth, irreducible, projective variety with $\deg(X)>1$ and let $\E$ be an Ulrich bundle on $X$.  
Then $h^0(\E^*)=0$.
\end{lemma}

\begin{proof} Assume by contradiction that $h^0(\E^*)>0$ and pick a  non-zero  section  \linebreak $s^* \in H^0(\E^*)$.  Since $h^0(\E)=\deg(X)\rk(\E) \geq \deg(X)>1$ by \cite[(3.1)]{be2}, we have $\E \not \cong \O_X$. It follows that $\coker(s^*) \neq 0$, and by dualizing we obtain
  \[ \xymatrix{0 \ar[r] &  \F:=\ker(s) \ar[r] & \E \ar[r]^{\hspace{-0.3cm s}} & \O_X,
    }
    \]
    where $\F \neq 0$. In particular, $\rk(\F)=\rk(\E)-1>0$ and $\im(s)$ is an ideal sheaf, whence
    \[ c_1(\F)=c_1(\E)-c_1(\im(s))=c_1(\E)+D,\]
    for an effective divisor $D$ on $X$. Thus,  denoting as usual by $\mu$ the \emph{slope} of a sheaf, we have 
    \[      \mu(\F)  =  \frac{c_1(\F) \cdot \O_X(1)^{\dim(X)-1}}{\rk(\F)}
       >  \frac{c_1(\E) \cdot \O_X(1)^{\dim(X)-1}}{\rk(\E)}=\mu(\E),
    \]
contradicting that $\E$, being Ulrich, is slope--semistable (\cite[Thm. 2.9(a)]{ch}). 
\end{proof}

\begin{remark}\label{rem:Pn} The assumption  $\deg(X) >1$  in Lemma \ref{lemma:dualnosez} is  essential,  as on $\PP^n$  trivial bundles are Ulrich, and in fact the only Ulrich bundles  (cf., e.g., \cite[Thm.\,2.3]{be2}). 
\end{remark}

\section{Curves associated to Ulrich bundles on threefolds}\label{S:corresp}

 In this section we will prove that the existence of Ulrich bundles on a smooth projective threefold is connected to the existence of smooth curves on the threefold with particular properties. Before giving the result, we fix some notation.  Let $X$ be a smooth projective threefold, $C \subset X$ a curve and $D$ a divisor on $X$. The short exact sequence
\begin{equation}\label{eq:villa}
 \xymatrix{
 0 \ar[r]   & \I_{C/X} \ar[r]  & \O_X\ar[r] & \O_C \ar[r] &  0 
}
\end{equation}
tensored by $\O_X(D+K_X)$  determines  a coboundary map
\[ d: H^1(C,\O_C(D+K_X)) \longrightarrow H^2(X,\I_{C/X}(D+K_X)), \]
whose dual, by Serre duality, is
\begin{equation}\label{eq:dstar}
  d^*: \Ext_X^1(\I_{C/X}(D),\O_X) \longrightarrow H^0(C,\omega_C(-D-K_X)).
\end{equation}
 Moreover, for any subspace \[
  W \subseteq \Ext_X^1(\I_{C/X}(D),\O_X) \cong 
  \left(H^2(X,\I_{C/X}(D+K_X))\right)^*\] we obtain a surjection
\begin{equation}
  \label{eq:alfa}
 \xymatrix{ H^2(X,\I_{C/X}(D+K_X)) \ar@{->>}[r] &  W^*}. 
\end{equation}
Since, by Serre duality,
\begin{eqnarray*}
   \Ext_X^1(\I_{C/X}(D),  W^* \* \O_X) \;\;\;  & \cong &  \Ext_X^1(\I_{C/X}(D+K_X),\omega_X) \* W^* \cong  \\
         \;\;\;\;\;\;\; \cong  H^2(\I_{C/X}(D+K_X))^*  \* W^*    & \cong  & \Hom_X(H^2(\I_{C/X}(D+K_X),W^*),
\end{eqnarray*}
we obtain an extension
\begin{equation}
  \label{eq:alfa2}
  \xymatrix{ 0 \ar[r] & W^* \* \O_X \ar[r] & \E \ar[r] & \I_{C/X}(D)   \ar[r] & 0}
  \end{equation}
  corresponding to the cocycle \eqref{eq:alfa}, that is, so that the coboundary map
  \[\xymatrix{H^2(\I_{C/X}(D+K_X)) \ar[r] & H^3(W^* \* \omega_X)\cong W^* \* H^3(\omega_X) \cong W^*}
  \]
  of \eqref{eq:alfa2} tensored by $\omega_X$ is the map \eqref{eq:alfa}.
 For any $\L \in \Pic(X)$, we may
twist \eqref{eq:alfa2} by $\L$ and  obtain a coboundary map 
\begin{equation}
  \label{eq:E0}
  \xymatrix{ H^2(X,\L(D) \* \I_{C/X}) \ar[r] & W^* \* H^3(X,\L)},
  \end{equation}
  which we will call {\it the coboundary map induced by $(W,\L)$}.

\begin{thm} \label{thm:codim2}
  Let  $X \subset \PP^m$  be a smooth projective threefold    with  $\deg(X)>1$  and $D$ be a divisor  on $X$.
  Let $r \geqslant 2$ be an integer.  Set $H:=\O_X(1)$.

  Then  there is a bijection between the set of Ulrich bundles $\E$  on $X$ satisfying  
\begin{equation}\label{eq:lops}
 \rk (\E)=r \; \; \mbox{and} \; \; \det (\E)=\O_X(D) 
 \end{equation}
  and   the set of pairs $(C,W)$, where $C \subset X$ is a smooth curve and $W \subseteq
  \Ext_X^1(\I_{C/X}(D),\O_X)$ is  an $(r-1)$--dimensional subspace, 
such that
\begin{eqnarray}
\label{eq:D} & d^*(W) \;\; \mbox{generates $\omega_C(-K_X-D)$ \;\;(cf. \eqref{eq:dstar}),} & \\
\label{eq:D1} & H^2 \cdot D   =  \frac{r}{2} \left(H^2 \cdot K_X+ 4 H^3\right), & \\
\label{eq:D2}
& H \cdot C   =   \frac{r}{12} \left(K_X^2\cdot H + c_2(X)\cdot H -  22  H^3\right) - \frac{1}{2} \left(H \cdot D \cdot K_X - H \cdot D^2 \right), & \\
\label{eq:D3}
             & g(C)  =  r H^3-r\chi(X,\O_X)-\frac{1}{6}D^3+\frac{1}{4}K_X \cdot D^2 \\
\nonumber & \hspace{2cm}  -\frac{1}{12} \left(K_X^2\cdot D + c_2(X)\cdot D\right)
+C \cdot D+1, & \\
  \label{eq:A} & h^0(K_X+3H-D)=0, \\
\label{eq:B} &   h^0(\I_{C/X}(D-H))=0, \\
\label{eq:C} & h^1(\I_{C/X}(D-pH))=0 \; \; \mbox{for} \; \; p \in \{1,2,3\}, \\
  \label{eq:E} & \mbox{the coboundary map} \; \delta: H^2(\I_{C/X}(D-3H)) \longrightarrow W^* \* H^3(\O_X(-3H)) \\
 \nonumber  & \mbox{induced by  $(W,\O_X(-3H))$  is either injective or surjective}.
\end{eqnarray}

Via the above bijection, $\E$ sits in  an  exact sequence
\begin{equation}\label{eq:drago}
 \xymatrix{
   0 \ar[r]   & W^*  \* \O_X \ar[r]  & \E \ar[r] & \I_{C/X}(D)  \ar[r] & 0.}
\end{equation}
\end{thm}

\color{black}  \begin{remark}\label{rem:c2} As $C$ is the $(r - 2)$--degeneracy locus of $u_{ r-1}$, it follows that its class of cohomology is $c_ 2 (E)$, see for instance \cite [(2.11)]{ott}.

We point out that the curve $C$ could be the 0--curve, i.e., $C$ could be empty. In this case $H\cdot C=0$, $g(C)=1$, $\mathcal I_{C/X}=\mathcal O_X$ and \eqref {eq:D} is vacuous. This case never happens if $D$ is big and nef, in particular if ${\rm Pic}(X)\cong \mathbb Z$, since in this case $ \Ext_X^1(\O_X(D),\O_X)=0$.

For examples where $C$ is the 0--curve, see \cite {LM}. 

\end{remark}  
\color{black}

\begin{proof}   Let $\E$ be an Ulrich bundle of rank $r$ on $X$ with $ \det (\E)=\O_X(D)$. Then $\E$ is globally generated  by, e.g.,  \cite[Thm.\;2.3(i)]{be2}.    We can 
pick a general $(r-1)$--dimensional subspace $V_{r-1} \subset H^0(\E)$ and let 
    $V_{r -1 } \* \O_X \stackrel{u_{r-1}}{\longrightarrow} \E$ be the evaluation morphism. It is well--known, see, e.g.,  \cite[Thm.~2.8]{ott} or \cite[Thm.~1]{ban}, that 
		$\coker(u_{r-1}) \cong \I_{C/X}(D)$ where $C$, the degeneracy locus of $u_{r-1}$, is a smooth curve, since $\dim ( X) = 3$.  Setting $W:=V_{r-1}^*$,  we thus have an exact sequence  like \eqref{eq:drago}.   We have $h^3(\E(K_X)) = h^0(\E^*) = 0$  by   Serre duality and   Lemma \ref{lemma:dualnosez}.   Therefore \eqref{eq:drago},  tensored by $\O_X(K_X)$, yields a surjection 
$H^2(\I_{C/X} (D+K_X)) \to\!\!\!\to W^*$,  whence, by duality   
\[ W  \subset \left(H^2(\I_{C/X}(D+K_X))\right)^* \cong \Ext_X^1(\I_{C/X}(D), \O_X). \] The  required correspondence associates to $\E$ the pair $(C,W)$.

 Note that,  dualizing  \eqref{eq:drago}, we find
\begin{equation}\label{eq:drago*}
\xymatrix{
0 \ar[r] & \O_X(-D) \ar[r] & \E^* \ar[r]  &  W  \* \O_X \ar[r] &  \omega_C(-K_X-D) \ar[r]  &  0, 
}
\end{equation}
 hence we have the surjection $W\* \O_X \longrightarrow \omega_C(-K_X-D)$ that proves \eqref{eq:D}.

Conversely, assume  we  have a smooth curve $C \subset X$ and an $(r-1)$--dimensional subspace $W \subseteq \Ext_X^1(\I_{C/X} (D),\O_X)$  
such that \eqref{eq:D} is satisfied.  As explained in the beginning of the section, this defines a  sheaf $\E$ on $X$ fitting in 
an exact sequence like \eqref{eq:drago}. 
 Tensoring \eqref{eq:villa}  by $\O_X(D)$ gives 
\[ \Shext_{\O_X}^i(\I_{C/X}(D),\O_X)=0 \;\; \; \mbox{for} \;\; i \geqslant2\]
and
\[ \Shext_{\O_X}^1(\I_{C/X}(D),\O_X) \cong  \Shext_{\O_X}^2(\O_C(D),\O_X) \cong \omega_C(-K_X-D).\]
From \eqref{eq:drago} we therefore find 
\[ \Shext_{\O_X}^i(\E,\O_X)=0, \;\mbox{for} \;\; i \geqslant2,\]
and
\[\Shext_{\O_X}^1(\E,\O_X) = \coker\Big[W \* \O_X \longrightarrow \omega_C(-K_X-D)\Big]=0,\]
by \eqref{eq:D}. Therefore, $\E$ is locally free on $X$.

Thus, to  finish the proof of  the theorem, we have left to show that if $\E$ is a vector bundle  verifying \eqref {eq:lops} and  fitting into an exact sequence like \eqref{eq:drago}, with $C \subset X$ a smooth curve,  $D$  a divisor  and  $W$ an  $(r - 1)$--dimensional subspace of $\Ext_X^1(\I_{C/X}(D),\O_X) $ satisfying \eqref{eq:D}, then $\E$ is Ulrich if and only if  \eqref{eq:D1}--\eqref{eq:E} hold.

Tensoring \eqref{eq:drago} by $\O_X(-pH)$, one sees that $\E$ is Ulrich if and only if
\begin{equation}
  \label{eq:A'}
  h^i(\I_{C/X} (D- pH))=0, \; \; \mbox{for} \; i \in \{0,1,3\}, \; p \in \{1,2,3\}
\end{equation}
and
\begin{eqnarray}
  \label{eq:E'}
  & \mbox{the coboundary maps} \; \delta_p: H^2(\I_{C/X} (D- pH)) \longrightarrow W^* \* H^3(\O_X(-pH)) & \\
\nonumber &   \mbox{are isomorphisms for} \; \; p \in \{1,2,3\}.&
\end{eqnarray}
  
  Condition \eqref{eq:A'} with $i=3$ is equivalent to $h^3(\O_X(D-pH))=0$ for $p \in \{1,2,3\}$, which is 
equivalent to
condition \eqref{eq:A} by Serre duality. 
Condition \eqref{eq:A'} with $i=0$ is equivalent to condition \eqref{eq:B}, whereas  \eqref{eq:A'} with $i=1$ is condition \eqref{eq:C}.

Consider now \eqref{eq:E'}. Conditions \eqref{eq:A}--\eqref{eq:C}, which as we just saw are equivalent to \eqref{eq:A'}, yield that the domain of $\delta_p$ has dimension 
$\chi(\I_{C/X}(D-pH))$. On the other hand, the target  of $\delta_p$  has dimension $(r-1)h^3(\O_X(-pH))=(1-r)\chi(\O_X(-pH))$ by Kodaira vanishing.
Thus, \eqref{eq:E'} is equivalent to the conditions 
\begin{eqnarray}
  \label{eq:E'1}
  & \chi(\I_{C/X}(D-pH))=(1-r)\chi(\O_X(-pH)) \;\; \mbox{for} \; \; p \in \{1,2,3\}, & \; \; \mbox{and} \\
\label{eq:E'2}   & \delta_p \; \; \mbox{is either injective or surjective} \; \; \mbox{for} \; \; p \in \{1,2,3\}.&
\end{eqnarray}
Suppose for a moment  that  condition \eqref {eq:E'1} is verified. For each $1 \leqslant p \leqslant 3$ we have
a commutative diagram
\[
  \xymatrix{
    H^2(\I_{C/X} (D- 3H)) \ar[r]^{\delta_3}  \ar[d]  & V_{r} \* H^3(\O_X(-3H)) \ar@{->>}[d]  \\
    H^2(\I_{C/X}(D-pH)) \ar[r]^{\delta_p}  & V_{r} \* H^3(\O_X(-pH)),  
    }
  \]
  which shows that  $\delta_1, \delta_2, \delta_3$ are surjective (or, equivalently,  injective) if and only if $\delta_3$ is surjective (or, equivalently,  injective).  
  Thus, condition \eqref{eq:E'2} is equivalent to \eqref{eq:E}. 

  Finally, consider condition \eqref{eq:E'1}.  It can be rewritten as 
	\begin{equation}\label{eq:chivarie}
	0 = - \chi(\O_X(D-pH)) + \chi(\O_C(D-pH) ) + (1-r)\chi(\O_X(-pH)), \;\; {\rm for} \;\; p \in \{1,2,3\}.
      \end{equation}
 Using Riemann--Roch on $X$ and $C$, the right hand side can be rewritten as

\begin{eqnarray*}
& \frac{r\,H^3}{6}\,p^3 + \left( \frac{r\,K_X \cdot H^2}{4} -  \frac{D\cdot H^2}{2} \right)\,p^2 + & \\
  & + \left( \frac{r}{12}(K_X^2 + c_2(X)) \cdot H  - \frac{H\cdot D \cdot K_X}{2} + \frac{H\cdot D^2}{2} - C\cdot H \right) \, p + & \\
&	\frac{1}{2} K_X \cdot D^2+ C\cdot D + 1 - g(C) - r \chi(\O_X)
  -  \frac{1}{6} D^3 - \frac{K_X\cdot D^2}{4} - \frac{1}{12}(K_X^2 + c_2(X)) \cdot D. &
  \end{eqnarray*}
	Dividing by the leading coefficient $\frac{r\,H^3}{6}$ gives a monic, cubic polynomial  in $p$,  which, by condition 
	\eqref{eq:E'1},  must coincide with
        $(p-1) (p-2) (p-3) = p^3 - 6 p^2 + 11 p -6$.
          
        Equating  the coefficients of  the terms of degrees $2$, $1$ and $0$ gives, respectively, \eqref{eq:D1}, \eqref{eq:D2} and
\eqref{eq:D3}.
\end{proof}

    Despite the  complexity of  conditions \eqref {eq:D}--\eqref {eq:E} in Theorem \ref {thm:codim2}, in certain cases some of them may be replaced by slightly easier ones, as shown in the following:

  \begin{lemma} \label{obs:codim2}  In the same setting as in Theorem \ref{thm:codim2},  we have:
    \begin{itemize}
    \item[(i)] If  $h^i(\O_X(-D))=0$ for $i \in\{1,2\}$, then $W$ can be identified with a subspace of $H^0(\omega_C(-K_X-D))$, and condition \eqref{eq:D} says that $W$ generates $\omega_C(-K_X-D)$.
    \item[(ii)] If  $h^i(\O_X(D-3H))=0$ for $i \in\{1,2\}$, then condition
      \eqref{eq:E} is equivalent to
      \begin{eqnarray}\label{eq:E*}
        & \mbox{the multiplication map} \; \nu: W \* H^0(\O_X(K_X+3H)) \longrightarrow H^0(\omega_C(3H-D)) & \\
 \nonumber   & \mbox{is either injective or surjective}. &
      \end{eqnarray}
    \end{itemize}
\end{lemma}

\begin{proof}
  (i) The assumptions yield that $d^*$ is an isomorphism, proving the assertion.

  (ii) The assumptions yield that the coboundary map  of \eqref{eq:villa}
  tensored by $\O_X(D-3H)$,  
\[ \gamma:H^1(\O_C(D-3H) \longrightarrow H^2(\I_{C/X}(D-3H)), \] is an isomorphism. The composed map
   $\delta \circ \gamma: H^1(\O_C(D-3H) \to V_{r-1} \* H^3(\O_X(-3H))$ is the dual of the multiplication map $\nu$, by Serre duality. Hence we see that condition \eqref{eq:E} is equivalent to \eqref{eq:E*}. 
\end{proof}

\begin{remark}\label{rem:smoothdiv}  Let $\E$ be an Ulrich bundle of rank $r$ sitting in an exact sequence like \eqref {eq:drago},  and set $V_{r-1}:=W^*$.   Since  $\E$ is generated by global sections  (\cite[Thm.\;2.3(i)]{be2}) and is nontrivial (Lemma \ref{lemma:dualnosez}), the divisor $D$ in Theorem \ref{thm:codim2} must be effective and nonzero (by Porteous). Pick  a general $r$-dimensional subspace $V_{r} \subseteq H^0(\E)$ containing
  $V_{r-1}$  and  consider the evaluation morphism $V_{r} \* \O_X \stackrel{u_{r}}{\longrightarrow} \E$. It is well--known (see again, e.g.,\,\cite[Thm.~2.8]{ott} or \cite[Thm.~1]{ban}) 
	that $\L:=\coker(u_{r})$ is supported on a member $S$ of $|\det\,(\E)|=|D|$ (the degeneracy locus of $u_{r}$), which is smooth as $\dim(X)=3$, and  that $\L$ 
 is locally free of  rank one  on $S$.  More precisely, by the Snake Lemma, we have a commutative diagram 
\begin{equation}\label{eq:diaguvbis}
 \xymatrix{
		&   & & 0 \ar[d] & \\
		& 0 \ar[d]& 0 \ar[d] & \O_X \ar[d] & \\
 0 \ar[r]   & V_{r-1}  \* \O_X \ar[r]^{\hspace{0.6cm}u_{r-1}} \ar[d] & \E\ar@{=}[d] \ar[r] & 
 \I_{C/X}(D) \ar[d] \ar[r] & 0 \\
0 \ar[r] & V_{r}  \* \O_X \ar[r]^{\hspace{0.2cm}u_{r}} \ar[d] &  \E \ar[r]  \ar[d]&  \L \ar[r]  \ar[d] &  0 \\
 & \O_X \ar[d]  &  0 &  0  &   \\
& 0 &  &  &
}
\end{equation}
from which we see that $C \subset S$ (corresponding to the  non--zero section of $ \I_{C/X}(D)$ appearing in the rightmost  column)  and $\L \cong \O_S(D-C)$.
 Dualizing the lower horisontal sequence in \eqref{eq:diaguvbis} 
we find 

\[
 \xymatrix{
0 \ar[r] & \E^* \ar[r]  &  V_{r}^*  \* \O_X \ar[r]  &   \Shext_{\O_X}^1(\L,\O_X) \cong   \O_S(C)  \ar[r]&  0,
 }
 \]which shows  that   $\O_S(C)$    is globally generated by $r$ sections. 
\end{remark}

 In the case of smooth embedded Fano threefolds, we get:

\begin{thm} \label{thm:fano} Let  $X \subset \PP^m$  be a smooth (Fano) threefold of degree $d>1$ with  \linebreak $\omega_X \cong \O_X(-\alpha)$, for $\alpha \in \{1,2,3\}$.
   Then $X$ carries an Ulrich bundle  $\E$ satisfying  
  \[ \rk (\E) = r \geqslant 2  \; \; \mbox{and} \; \; \det (\E) =  \O_X(m),  \; \; \mbox{with} \; \; m \in \ZZ^+, \]
  if and only if
\begin{eqnarray}
  \label{eq:fano1} & m=\frac{r}{2} (4 - \alpha)
                     \end{eqnarray}
and
there exists a smooth curve $C \subset X$ such that
  \begin{eqnarray}
  \label{eq:fano2} &   \deg(C) = \frac{r d}{12} (\alpha^2 - 22)  + \frac{2 r}{\alpha} + \frac{md}{2}(\alpha + m) \\
     \label{eq:fano3} &  g(C)= r(d-1)+\frac{1}{3}m^3d+\frac{1}{4}\alpha m^2d+\frac{1}{12}\left[r(\alpha^2-22)-\alpha^2\right]md+\frac{2(r-1)}{\alpha}m+1, \\ 
\label{eq:fano4} &\omega_C (\alpha-m) \; \; \mbox{is generated by  an   $(r-1)$-dimensional space} \; \; W, \\
    \label{eq:fano6} &  \mbox{the multiplication map} \; \; \\
  \nonumber &  \nu: W \* H^0(\O_X(3-\alpha)) \longrightarrow H^0(\omega_C(3-m)) \\
    \nonumber &  \mbox{ is either injective or surjective}, \\
    \label{eq:fano5} &  \mbox{the restriction map} \; \; H^0(X,\O_X(m-p)) \longrightarrow H^0(C,\O_C(m-p)) \\
    \nonumber &  \mbox{is an isomorphism for} \; \; p \in \{1,2,3\}.
\end{eqnarray}
\end{thm}

\begin{proof}
  We apply Theorem \ref{thm:codim2}.
  Using Lemma \ref{obs:codim2}(i), we see that condition \eqref{eq:fano4} is \eqref{eq:D}. 
  Conditions \eqref{eq:D1}, \eqref{eq:D2} and \eqref{eq:D3} in Theorem \ref{thm:codim2} are, respectively,  
\eqref{eq:fano1}, \eqref{eq:fano2} and \eqref{eq:fano3}. (To  compute  $c_2(X) \cdot H$, we use    
  $1=\chi(X, \mathcal O_X) = \frac{\alpha}{24} c_2 \cdot H$ by, e.g.,
  \cite[Exc. 15.2.5]{Fu}.)

Now consider condition \eqref{eq:A} in Theorem \ref{thm:codim2}. It reads  
 $h^0(\O_X(3-\alpha-m))=0$, which by \eqref{eq:fano1} is equivalent to $3-\alpha-\frac{r}{2}\left(4-\alpha\right) <0$,  which  is  clearly  satisfied.

  Next consider conditions \eqref{eq:B}--\eqref{eq:C} in Theorem \ref{thm:codim2}. Together they  are equivalent to 
\begin{equation} \label{eq:fanvan}
  h^i(X,\I_{C/X}(m-p))=0 \; \; \mbox{for} \; \; i \in \{0,1\}, \; p \in \{1,2,3\}.
  \end{equation}
From the short exact sequence
\[
  \xymatrix{
0 \ar[r] & \I_{C/X}(m-p) \ar[r] & \O_X(m-p) \ar[r] & \O_C(m-p) \ar[r] & 0, 
    }
  \]
  and the fact that $h^1(\O_X(m-p))=0$,
 we see that \eqref{eq:fanvan} is equivalent to \eqref{eq:fano5}.

 Finally,  by  Lemma \ref{obs:codim2}(ii),  condition  \eqref{eq:E} in Theorem \ref{thm:codim2}
can be restated as \eqref{eq:fano6}.
\end{proof}

\begin{remark} \label{rem:fano}  Suppose $\E$ is an Ulrich  bundle on a   smooth Fano  threefold  $X$ as  in  Theorem \ref {thm:fano}. 
 Then  $h^i(\E^*)=h^{3-i}(\E(K_X))=h^{3-i}(\E(-\alpha))=0$ for all $0 \leqslant i \leqslant 3$,   so that  the cohomology of \eqref{eq:drago*}  yields  $W=H^0(\omega_C(\alpha- m))$.  
\end{remark}

 We will make use of Theorem \ref{thm:fano} in the proof of   our main results   in the next section. In the rest of this section we give some consequences and examples.

\begin{corollary} \label{cor1}
 Let $X \subset \PP^{g+1}$ be a smooth Fano threefold of degree $2g-2$ such that $\omega_X\cong \O_X(-1)$ and $\Pic(X) \cong \ZZ[\O_X(1)]$. Then $X$ carries no Ulrich bundles of odd rank.   
\end{corollary}

\begin{proof}
It is well--known that varieties of degree  $d > 1$ with Picard group generated by the hyperplane class cannot carry Ulrich line bundles (see, e.g., \cite[\S\;4]{be2}). As for rank $r\geq 2$, Theorem \ref{thm:fano} (with $\alpha=1$) yields that $m=\frac{3r}{2}$ must be an integer, cf.\;\eqref{eq:fano1}. 
\end{proof}

Recall that Fano threefolds satisfying the hypotheses of the corollary (called {\it prime Fano threefolds of index one}) are classified and that $g \in \{3,\ldots,10,12\}$ (see \cite {IP}). In \cite{CFK-ind1} we  prove that all such Fano threefolds carry Ulrich bundles of all even ranks.

\begin{example} \label{ex:fanoindex1-rk2}
  Let $X \subset \PP^{g+1}$ be a  smooth  Fano threefold of degree $2g-2$ such that  $\omega_X \cong \O_X(-1)$.
  Theorem \ref{thm:fano} yields that $X$ carries a rank--two Ulrich bundle $\E$ such that
$\det (\E) =  \O_X(m)$ 
if  and only if $m=3$ and there exists a smooth curve $C \subset X$ of degree $5g-1$ such that $\omega_C \cong \O_C(2)$ (in particular $g(C)= 5g$) and $C \subset \PP^{g+1}$ is non--degenerate and linearly and quadratically normal.

 Since  in \cite{CFK-ind1} we  prove the existence of such rank--two Ulrich bundles in the cases where $\Pic(X) \cong   \ZZ[\O_X(1)]  $, the existence of such a curve will follow as a consequence. 
\end{example}

  The following proves Theorem \ref{thm:veronese}(i):  

\begin{corollary} \label{cor2}
Let $X \subset \PP^9$ be the $2$--Veronese embedding of $\PP^3$. Then $X$ carries no Ulrich bundles of odd rank.     
\end{corollary}

\begin{proof}
  We have $\omega_X \cong \O_X(-2)$. If there exists an Ulrich bundle of rank $r$ on $X$, then Theorem \ref{thm:fano} (with $\alpha=2$) yields the existence of a  curve of degree $r(4r-3)$ (since $m=r$ by \eqref{eq:fano1}). As $\O_X(1)$ is $2$--divisible, the degree must be even, whence $r$ must be even. 
\end{proof}

\begin{example} \label{ex:fanoindex2}
  Let $X \subset \PP^{d+1}$ be a  smooth  Fano threefold of degree $d$  such that \linebreak $\omega_X \cong \O_X(-2)$. 

  (i) Theorem \ref{thm:fano} yields that  $X$ carries a rank-two Ulrich bundle $\E$ such that \linebreak 
$\det (\E)  \cong \O_X(m)$  
if  and only if $m=2$ and there exists a smooth elliptic normal curve $C \subset X$ of degree $d+2$.  The curve $C$ is contained in a smooth surface $S \in |\O_X(2)|$ (cf. Remark \ref{rem:smoothdiv}), which is a  $K3$ surface. 

 The existence of rank--two Ulrich bundles on $X$ was proved 
in \cite[Prop. 6.1]{be2} by proving the existence of such an elliptic curve (cf. Proposition \ref{prop:beauville} below). (The existence in the case of cubic threefolds in $\PP^4$ had previously been proved in \cite[Prop. 5.1]{ch}.)

(ii) Theorem \ref{thm:fano} yields that  
$X$ carries a rank--three Ulrich bundle $\E$  such that \linebreak 
$\det (\E)  \cong \O_X(m)$ 
if  and only if $m=3$ and there exists a smooth curve $C \subset X$
of degree $3d+3$ and genus $2d+4$ such that $\omega_C(-1)$ is globally generated (by two sections), $C \subset \PP^{d+1}$ is  linearly  and quadratically normal and  the Petri map
\begin{equation} \label{eq:petri} \mu_{0,\O_X(1)}: H^0(\O_C(1)) \* H^0(\omega_C(-1)) \longrightarrow H^0(\omega_C)
  \end{equation}
is injective. (Indeed, by \eqref{eq:fano5} and Remark \ref{rem:fano} the Petri map equals the map $\nu$ in \eqref{eq:fano6}.) 
The curve $C$ is contained in a smooth surface $S \in |\O_X(3)|$ (cf. Remark \ref{rem:smoothdiv}), which is a  canonical surface, i.e., $\omega_{S} \cong \O_{S}(1)$. 

 In the  case where $X$ is a {\it general}  smooth cubic threefold   in $\PP^4$, i.e., $d=3$, the existence of rank--three Ulrich bundles on $X$ was proved in \cite[Prop. 5.4]{ch} by constructing a curve $C$ satisfying the above conditions using {\tt Macaulay2} (cf. \cite[App. A.2]{ch}). By contrast, we can deduce the existence of such a curve $C$ as a consequence of  Theorem \ref{thm:Fanoind2}.
\end{example}

\section{Ulrich bundles on   Del Pezzo threefolds }\label{S:bundles}

 This section is devoted to the proofs of Theorems \ref{thm:Fanoind2} and   \ref{thm:veronese}(ii), recalling that part (i) of the latter has been proved in 
Corollary \ref{cor2}.   

We recall that a smooth  Fano threefold $X$ has    {\it even index} if there exists an ample   divisor  $H$ on $X$ such that $K_X = - 2H$. If $H$ is very ample,  it is well--known that  it embeds $X$ as a subvariety of degree $d = H^3$ in  $\PP^{d+1}$.  The  general curve section of $X$  is an 
elliptic curve and the general hyperplane section is a smooth Del Pezzo  surface.   For this reason, such threefolds are also called {\it Del Pezzo threefolds}. 
The {\it index} of $X$ is the order of divisibility of $K_X$ in $\Pic(X)$.   By   the  classification of such varieties (see, e.g., \cite{IP}),   there is   one case of index 
four, namely the $2$--Veronese embedding of $\PP^3$ in $\PP^9$ (where $d=8$), and the following cases of index two:

\begin{itemize}
\item[$d=3$:] $X$ is a smooth cubic hypersurface in $\mathbb P^4$, with $\Pic(X) = \mathbb{Z}[H]$; 
\item[$d=4$:] $X$ is a smooth complete intersection of type $(2,2)$ in $\mathbb P^5$, with $\Pic(X) = \mathbb{Z}[H]$;
\item[$d=5$:] $X$ is a  smooth   section with a linear space $\Lambda \cong \mathbb P^6$  of the Grassmannian $\mathbb{G}(1,4)$ sitting in $\mathbb{P}^9$ via  its   Pl\"ucker embedding; one has ${\rm Pic}(X) = \mathbb{Z}[H]$. 
\item[$d=6$:] two cases occur: either $X$ is the Segre embedding of  $\PP^1 \x 
\PP^1 \x \PP^1$ in $\PP^7$, in which case $\Pic(X) \cong \mathbb{Z}^3$, or $X$ is a  smooth  hyperplane section 
of the Segre embedding of $\mathbb P^2 \times \mathbb P^2$ in $\mathbb P^8$, and $\Pic(X) \cong \mathbb{Z}^2$; 
\item[ $d=7$:]  $X$  is isomorphic to the blow--up of $\mathbb P^3$ at a point $p$,  embedded in $\mathbb P^8$ via  the proper transform on $X$ of the linear system 
of quadrics in $\mathbb P^3$ passing through the point $p$. In this case one has $\Pic(X) \cong \mathbb{Z}^2$.  
\end{itemize}


We will need the following lemma.

\begin{lemma} \label{lemma:vaniz}   Let  $X \subset \PP^{d+1}$ be a smooth    Del Pezzo threefold   of degree $d$. 

\begin{itemize}
\item[(i)] If $\U$ is an Ulrich bundle on  $X$  and  $\E$ is any vector bundle on $X$ satisfying $h^0(\E(-1))=h^1(\E(-2))=0$, then 
	$h^2(\U \* \E^*)=h^3(\U \* \E^*)=0$.   In particular, for any Ulrich bundle $\U$ on $X$ one has $h^2(\U \* \U^*)=h^3(\U \* \U^*)=0$. 
 \item[(ii)]   If $\E_i$, $i \in \{1,2\}$, are vector bundles on $X$  satisfying
  \begin{equation} \label{eq:Milan}
    \rk(\E_i)=r_i, \;\;  c_1(\E_i)=[\O_X(r_i)] \;\; \mbox{and} \;\; c_2(\E_i) \cdot \O_X(1)=\frac{r_i}{2}\left(r_i d-d+2\right),
  \end{equation}
then 
$\chi(\E_1 \* \E_2^*)=-r_1r_2$.  
\end{itemize}
\end{lemma}

\begin{proof}
	 (i) Since $\U$ is Ulrich, there exists a  resolution of the form 
\[ \xymatrix{ \cdots \ar[r] & \O_{\PP^{d+1}}(-2)^{\+ a_2}  \ar[r] & \O_{\PP^{d+1}}(-1)^{\+ a_1}   \ar[r] & \O_{\PP^{d+1}}^{\+ a_0} \ar[r] & \U  \ar[r] & 0,}\]
 where $a_i \in \ZZ^+$ (cf. e.g. \cite[Thm.\;1]{be2}).  Tensoring by $\E^*$, we obtain
      \[ \xymatrix{ \cdots \ar[r] & \E^*(-2)^{\+ a_2}  \ar[r] & \E^*(-1)^{\+a_1}   \ar[r] &
          {\E^*}^{\+ a_0} \ar[r] & \U \* \E^*  \ar[r] & 0}\]
      This implies, by Serre duality and the given assumptions, that
      \begin{eqnarray*}
			h^3(\U \* \E^*) \leqslant h^3({\E^*})^{\+ a_0}=a_0h^3(\E^*)=a_0h^0(\E(-2)) \leqslant a_0h^0(\E(-1))=0 ,  \\
h^2(\U \* \E^*)\leqslant a_0h^2(\E^*)+a_1h^3(\E^*(-1))=a_0h^1(\E(-2))+a_1h^0(\E(-1))=0.
\end{eqnarray*}

(ii)    We borrow an argument from \cite[Proof of Prop.\;5.6]{ch}. Set $\E'_i:=\E_i^*(2)$. Then one may check that
properties \eqref{eq:Milan} hold with $\E_i$ replaced by $\E_i'$, whence by Riemann--Roch 
\[
  \chi(\E_1 \* \E_2^*(-1))=\chi(\E'_1 \* {\E'_2}^*(-1)).
\]
At the same time, Serre duality gives
\[
\chi(\E_1 \* \E_2^*(-1))=-\chi(\E_1^* \* \E_2(-1))=-\chi(\E'_1 \* {\E'_2}^*(-1)),
\]
so that we must have $\chi(\E_1 \* \E_2^*(-1))=0$. 
Taking a general $S \in |\O_X(1)|$, which is a Del Pezzo surface of degree $d$,
the restriction sequence thus yields  $\chi(\E_1 \* \E_2^*)=\chi(\E_1|_S \* \E_2^*|_S)$. For  vector bundles $\U$ and $\V$ on a surface $S$,  the  Riemann-Roch  theorem  gives
\begin{eqnarray*} \chi (\U \* \V^*) = \frac{1}{2}\left(\rk(\V)c_1(\U)^2+\rk(\U)c_1(\V)^2\right)-
  \rk(\V)c_2(\U)-\rk(\U)c_2(\V) \\
  -c_1(\U)c_1(\V) 
-\frac{1}{2}\left(\rk(\V)c_1(\U)-\rk(\U)c_1(\V)\right) \cdot K_S + \rk(\U)\rk(\V)\chi(\O_S).
\end{eqnarray*}
Inserting $\U=\E_1|_S$ and $\V=\E_2|_S$, and using that $\rk(\E_i|_S)=r_i$, $\det(\E_i|_S)=\det(\E_i)|_S = \O_S(r_i)$ and $c_2(\E_i|_S)=c_2(\E_i) \cdot S=\frac{r_i}{2}\left(r_i d-d+2\right)$ by \eqref{eq:Milan}, the result follows.  
\end{proof}

The  following settles  all statements but the  nonemptiness  claims in Theorems \ref{thm:Fanoind2}   and \ref{thm:veronese}(ii), recalling that stable bundles are simple:

 \begin{proposition} \label{prop:dimmod}  Let  $ X\subset \PP^{d+1}$ be a smooth    Del Pezzo threefold of degree $d$. 
If nonempty, the moduli space of simple rank--$r$ Ulrich bundles with determinant $\O_X(r)$ on  $X$ is   smooth of dimension $r^2+1$.   
 \end{proposition} 

\begin{proof}   Let $\E$ be a member of the moduli space in question. By simpleness $h^0(\E \* \E^*)=1$; moreover   
$h^2(\E \* \E^*)= h^3(\E \* \E^*)=0$ by Lemma \ref{lemma:vaniz}(i).   It is then well-known that   the moduli space 
  is smooth at $[\E]$ (see, e.g., \cite[Prop. 2.10]{ch}), of dimension
\[ h^1(\E \* \E^*)= -\chi(\E \* \E^*)+h^0(\E \* \E^*)
    =r^2+h^0(\E \* \E^*) = r^2+1,
\]using Lemma \ref{lemma:vaniz}(ii).  
\end{proof}

We have left to prove the existence parts of Theorems \ref{thm:Fanoind2}   and \ref{thm:veronese}(ii). Since both Ulrichness and stability are open conditions, it suffices 
to prove the existence of one stable Ulrich bundle on $X$ (this observation had been done already in \cite[Lemma\,2.3]{CFaM1}).  This will be proved by induction on the rank $r$, starting with the case $r=2$.

 \subsection{  The cases of rank $r=2$ }\label{S:rango2} 

The next result is due to Beauville \cite{be2}; we  will briefly recall  part of  the proof, because we will  need it later. 

\begin{proposition}[Beauville] \label{prop:beauville}  Let  $ X\subset \PP^{d+1}$ be a smooth   Del Pezzo threefold    of degree $d$.  
There exists a rank-two Ulrich bundle on  $X$   with  determinant $\O_X(2)$. 
Moreover,  the moduli space parametrizing such bundles  is $5$--dimensional. 
\end{proposition}

\begin{proof} By  Example \ref{ex:fanoindex2}(i) it suffices to prove that $X$ contains a smooth elliptic normal curve of degree $d+2$.   This is done 
in \cite[Lemma 6.2]{be2}. We recall the proof when $d \leq 7$:   take  a general $S \in |\O_X(1)|$. This is a  smooth  Del Pezzo surface, obtained by blowing up $9-d$ points on $\PP^2$ in general position. Consider the linear system on $S$ corresponding to quartic curves in $\PP^2$ passing doubly through two of the blown-up points and simply through the rest. This is a $(d+1)$-dimensional base point free linear system, whose general member is a smooth elliptic curve $E_0$ of  degree  $d+1$. For any point $p \in E_0$, consider any line $\ell_0$ in $X$ through $p$ not lying on $S$. As $\ell_0 \cdot S=1$, the curves $E_0$ and $\ell_0$ intersect transversely at one point. In 
\cite[Pf.\;of\;Lemma\;6.2]{be2} it is proved that $E_0 \cup \ell_0$ deforms to a smooth elliptic curve $E$ in $X$, which is the desired curve.

The statement about dimension of  the moduli space   is proved in \cite[Prop. 6.4]{be2}.
\end{proof}

  The existence of a {\it stable} Ulrich  bundle is proved in the following result, recalling that slope--stability implies stability.  

  \begin{proposition} \label{prop:beauville-stable}  Let
 $ X\subset \PP^{d+1}$ be a smooth    Del Pezzo threefold   of degree $d$. There exists a rank-two  slope--stable Ulrich bundle on  $X$ with   determinant $\O_X(2)$.
\end{proposition}

\begin{proof} Let $\E$ be a rank--two  Ulrich bundle on $X$ with determinant $\O_X(2)$. If $\E$ is not slope--stable, we have a destabilizing sequence
  \begin{equation} \label{eq:destab-st} \xymatrix{
0 \ar[r] & \F \ar[r] & \E \ar[r] & \G \ar[r] & 0,
    } \end{equation}
  with $\F$ and $\G$ of rank one. By saturating the sequence, we may assume that $\G$ is torsion-free. It follows from \cite[Thm. 2.9]{ch} that
  both $\F$ and $\G$ are Ulrich  line  bundles of slope $\mu(\E)=d$ (see also \cite[(3.2)]{be2}).
    
   In   the cases  where      $\Pic(X) \cong \ZZ[\O_X(1)]$, that is,  when  $d \in \{3,4,5\}$, 
	there exist no Ulrich line bundles on $X$  (see again \cite[\S 4]{be2}), so  we have a contradiction. 
	  There are again no Ulrich line bundles on $X$ when $d=7$ (cf. \cite[Cor. 2.5 and Prop. 2.7]{CFiM}) and $d=8$ (cf. Corollary \ref{cor2}),   
	and we have the same contradiction.  We have left to treat the two cases with $d=6$. 
	
	If $X$ is a  smooth  hyperplane section of $\PP^2 \x \PP^2$ in its Segre embedding in $\PP^8$,  the only Ulrich line bundles on  $X$  are  $\pi_i^*\O_{\PP^2}(2)$, $i=1,2$, where $\pi_i:X \to \PP^2$ are the projections (cf.  \cite[Cor. 2.7]{CFaM2}).
Thus, in \eqref{eq:destab-st} we have $\F\cong \pi_1^*\O_{\PP^2}(2)$ and $\G \cong \pi_2^*\O_{\PP^2}(2)$, or vice--versa. Since
  \[ \dim (\Ext^1_X(\pi_1^*\O_{\PP^2}(2),\pi_2^*\O_{\PP^2}(2)))=3, \] 
  and similarly with $\pi_1$ and $\pi_2$ interchanged (cf. \cite[Example\;5.4]{CFaM2}), bundles which are not slope--stable  fill up  a locus of  dimension at most $2$  in the moduli space, which is $5$--dimensional   by the last statement in Proposition \ref{prop:beauville}. 
   Hence there exist slope--stable rank--two  Ulrich bundles on $X$ with determinant $\O_X(2)$. 
  
  If $X \cong \PP^1 \x \PP^1 \x \PP^1$, one may similarly check that the only possibilities for $\F$ and $\G$ in \eqref{eq:destab-st} are
  $\pi_1^*\O_{\PP^1}(2) \* \pi_3^*\O_{\PP^1}(1)$ and $\pi_2^*\O_{\PP^1}(2) \* \pi_3^*\O_{\PP^1}(1)$, where $\pi_i:X \to \PP^1$ are the three projections, up to permutations (cf.  \cite[Prop. 3.2]{CFaM1}). Since  \[ \dim (\Ext^1_X(\pi_2^*\O_{\PP^1}(2) \* \pi_3^*\O_{\PP^1}(1), \pi_1^*\O_{\PP^1}(2) \* \pi_3^*\O_{\PP^1}(1)))=3\]
  (cf. \cite[Proof of Prop. 6.5]{CFaM1}), once again bundles which are not slope--stable  fill up  a locus of  dimension at most $2$,  and we are done as above.     
\end{proof}

  This concludes the proofs of Theorems \ref{thm:Fanoind2}   and \ref{thm:veronese} in the case $r=2$.  

\begin{remark} \label{rem:duecomp}
   In the  cases   $d \in\{6,7\}$ the moduli space of rank--two Ulrich bundles with determinant $\O_X(2)$ is reducible. 
	For $X\cong \PP^1 \x \PP^1 \x \PP^1$, by \cite[Thm.B(3)]{CFaM1} and taking into account possible permutations of the generators $h_1, h_2, h_3$ in the expressions of $c_2(\E)$,  the moduli space  consists of six irreducible components of dimension $5$, three of which are generically smooth containing just one point representing the $S$--equivalence class of strictly slope--semistable Ulrich bundles whereas the general point corresponds to a slope--stable Ulrich bundle,	and three components are smooth contaning only slope--stable Ulrich bundles. Reducibility  also occurs  for a  smooth  hyperplane section of the Segre variety 	$\mathbb P^2 \x \mathbb P^2$, cf. \cite[Thm.\;5.6]{CFaM2}, and for the case $d=7$, cf. \cite[Thm.B(1)]{CFiM}, where in both cases the components 
	are characterized by different $c_2(\E)$. By contrast, when $d=3$, the case when $X \subset \PP^4$ is a  cubic threefold,  the moduli space is birational to the intermediate Jacobian of $X$, and is thus irreducible (cf. \cite{be1}).   
\end{remark}

 \subsection{Conclusion of the proof of Theorem \ref{thm:Fanoind2} }\label{S:fanoind2} 

Let $ X\subset \PP^{d+1}$ be a smooth   Fano threefold   of degree $d$ and   index two with $\omega_X \cong \O_X(-2)$. 
 Denote  by $\H_1(X)$ the Hilbert scheme of lines contained in $X$, which by  \cite[Prop.\;3.3.5(i)]{IP}  is smooth, projective, of pure dimension two.
 (It is irreducible for $3 \leqslant d \leqslant 5$ (cf. \cite[Thm.\;1.1.1]{KPS}) and reducible for $d=6,7$ (cf.\;\cite[Prop.\;3.5.6]{IP}, but we will not need this.)

\begin{lemma} \label{lemma:puntiindip-base}  Let
 $ X\subset \PP^{d+1}$ be a smooth   Fano threefold   of degree $d$ and   index two with $\omega_X \cong \O_X(-2)$.  
  Let  $E \subset S_2\in |\O_X(2)|$  be a general curve and a general surface associated to a rank-two Ulrich bundle  on $X$  as in Theorem \ref{thm:fano} and Remark \ref{rem:smoothdiv}, and let $\ell$ be a general line in (a component of) $\H_1(X)$. 
	Then the restriction map induces an isomorphism
  \[ \xymatrix{H^0(\O_{S_2}(E)) \ar[r]^{\cong} & H^0(\O_{\ell \cap S_2}(E)).}\] 
\end{lemma}

\begin{proof} Consider the curve $E_0$ contained in the smooth hyperplane section $S$ of $X$ and the line $\ell_0$  as in   the proof of Proposition \ref{prop:beauville}. 
Then $\ell_0$ is contained in a hyperplane section $S' \neq S$ of $X$. For $\ell \subset X$ general line in (a component of) $\H_1(X)$, 
we have $\ell \cap \ell_0 = \emptyset$ and $\ell \cap S$ is one point. Hence
  $\ell$ intersects $E_0 \cup \ell_0$ in at most one point, and the same is true replacing $E_0$ with any curve linearly equivalent to it on $S$.  Therefore, as one deforms the pair $(E_0 \cup \ell_0,S \cup S')$ to a pair $(E,S_2)$, where $S_2 \in |\O_X(2)|$ is a smooth, irreducible  $K3$ surface and $E \subset S_2$ is a smooth elliptic normal curve of degree $d+2$, the line $\ell$ intersects any member of the pencil $|E|$ on $S_2$ in at most one point. Since $\ell \cap S_2$ consists of two points, we have $H^0(\I_{(\ell \cap S_2)/S_2}(E))=0$. Thus, the restriction map $H^0(\O_{S_2}(E)) \to H^0(\O_{\ell \cap S_2}(E))$ is injective. As both spaces are two-dimensional, the map is an isomorphism.
  \end{proof}
  
  Using induction on $r$ we will prove the following statement, which concludes the proof of (the nonemptiness  part of) Theorem \ref{thm:Fanoind2}.  

\begin{proposition} \label{prop:indrho}  Let
 $ X\subset \PP^{d+1}$ be a smooth   Fano threefold   of degree $d$ and   index two such that   $\omega_X \cong \O_X(-2)$.  
 For every integer $r \geqslant 2$ there exists an irreducible component 
   $ \UU_r $ of the moduli space of   slope--semistable   vector bundles of rank $r$ and determinant $\O_X(r)$ on $X$ such that its general member is a slope--stable Ulrich bundle.   Moreover, for $r \geqslant 3$,  $\UU_r$ contains a closed subscheme $\UU_r^{\rm {ext}}$  of dimension at most $r^2-r +2$  consisting of extensions $\E$ of the form
\begin{equation*}\label{eq:Ext}
 \xymatrix{
 0 \ar[r]   & \E_{r-1} \ar[r] & \E \ar[r] & \I_{\ell/X}(1) \ar[r] & 0, 
}
\end{equation*}
where $\E_{r-1}$ is a general member of $\UU_{r-1}$ and $\ell$ is a general line in (a component of) $\H_1(X)$. 
\end{proposition}

The case $r=2$ has been proved in Propositions \ref{prop:beauville} and
\ref{prop:beauville-stable}. Assume now  we have proved Proposition \ref{prop:indrho} for some $r \geqslant 2$. We will prove that it also holds for $r+1$.

We denote by $\E_{r}$ a general member of $\UU_r$. Since $\E_{r}$ is Ulrich, by Theorem  \ref{thm:fano}  and \eqref{eq:drago} it sits in a short exact sequence
\begin{equation*}\label{eq:Erho}
 \xymatrix{
 0 \ar[r]   & \CC^{r-1}  \* \O_X \ar[r] & \E_{r} \ar[r] & \I_{C_{r}/X}(r) \ar[r] & 0, 
}
\end{equation*}
where $C_{r}$ is a smooth irreducible curve in $X$ numerically equivalent to $c_2(\E_{r})$,   cf. Remark \ref{rem:c2},   with
 \begin{equation}
   \label{eq:fanoindex2-2}  \deg (C_{r})  =  \frac{r}{2}\left(r d-d+2\right).
 \end{equation}

Pick a general line in (a component of) $\H_1(X)$. 

\begin{lemma} \label{lemma:Egen-1}
  We have
  \begin{itemize}
  \item[(i)] $\E_{r}|_\ell \cong \O_{\PP^1}(1)^{\+r}$,
\item[(ii)] $h^2(\E_{r}^*(-1) \*\I_{\ell/X})=r$.
  \end{itemize}
\end{lemma}

\begin{proof}
  (i) Recall from Theorem \ref{thm:codim2} and Remark \ref{rem:smoothdiv} that there is a smooth surface $S_{r} \in |\O_X(r)|$ containing $C_{r}$ and an exact sequence
  \begin{equation}\label{eq:diagrho}
 \xymatrix{
 0 \ar[r]   & \E_{r}^*\ar[r] & \O_X^{\+r}  \ar[r] &   \O_{S_{r}}(C_{r}) \ar[r]  & 0.
}
\end{equation}
Restricting to $\ell$, and setting $Z_{r}:=\ell \cap S_{r}$ (a scheme of length  $r$), we obtain
\begin{equation} \label{eq:restaell}
  \xymatrix{
  0 \ar[r]   & \E_{r}^* \ar[r] \ar[d] & \O_X^{\+ r} \ar[r] \ar[d] &   \O_{S_{r}}(C_{r}) \ar[r]  \ar[d] & 0 \\
  0 \ar[r]   & \E_{r}^*|_{\ell}\ar[r] \ar[d] & \O_{\ell}^{\+ r} \ar[r]  \ar[d] &   \O_{Z_{r}}(C_{r}) \ar[r]  \ar[d] & 0 \\
  & 0 & 0 & 0
}
\end{equation}
 (The  map $\E_{r}^*|_{\ell}\longrightarrow  \O_{\ell}^{\+ r}$ is injective because $\E_{r}^*|_{\ell}$ is locally free on $\ell$ and therefore it does not contain torsion subsheaves.) 
Since $\E_{r}$ is Ulrich, we have $H^i(\E_{r}^*)=H^{3-i}(\E_{r}(-2))=0$ for all $i$, whence we have a commutative diagram
\[
  \xymatrix{
   H^0(\O_X^{\+ r}) \ar[r]^{\hspace{-0.35cm}\cong} \ar[d] &   H^0(\O_{S_{r}}(C_{r}))   \ar[d]^f  \\
H^0(\O_{\ell}^{\+ r}) \ar[r]^{\hspace{-0.4cm}g} & H^0(\O_{Z_{r}}(C_{r})), 
}\]
where all four spaces are isomorphic to $\CC^{r}$. Grant for the moment the following:
\begin{equation}
  \label{eq:grant}
  \mbox{$f$ is injective.}
\end{equation}
Then $f$ is an isomorphism, and it follows that the map $g$ must be surjective, whence an isomorphism as well. Going back to \eqref{eq:restaell}, we see that
$h^i(\E_{r}^*|_{\ell})=0$ for $i \in \{0,1\}$. Since $\deg(\E_{r}^*|_{\ell})=-c_1(\E_{r})\cdot \ell =-\O_X(r) \cdot \ell=-r$, then
$\E_{r}^*|_\ell \cong \O_{\PP^1}(-1)^{\+r}$, proving (i).

We now prove \eqref{eq:grant},  which is equivalent to
\begin{equation}
  \label{eq:grant'}
  H^0(\I_{Z_{r}/S_{r}} ( C_{r}) )=0.
\end{equation}
The case $r=2$ follows from Lemma \ref{lemma:puntiindip-base}, with $C_2 = E$.

If $r \geqslant 3$, we may use Proposition \ref{prop:indrho} (recalling that by the induction hypothesis we are assuming that it holds for $r$) and specialize $\E_{r}$  in a one-parameter  flat  family over the disc $\DD$  to a vector bundle $\E'_{r}$ sitting in an extension
\begin{equation} \label{eq:sittinginext}
  \xymatrix{
 0 \ar[r]   & \E_{r-1} \ar[r] & \E'_{r} \ar[r] & \I_{\ell'/X}(1) \ar[r] & 0 
}
\end{equation}
with $\E_{r-1}$ general in $\UU_{r-1}$ and $\ell'$ a general line in (a component of) $\H_1(X)$.
 Then the surface $S_{r}$ in \eqref{eq:diagrho} specializes to $S_{r-1} \cup S_1$, where $S_{r-1} \in |\O_X(r-1)|$ and $S_1 \in |\O_X(1)|$
are smooth surfaces. In this process  $Z_{r}$, the intersection scheme of $\ell$ with  $S_{r}$, specializes to  
the disjoint union  $Z_{r-1} \cup x$, where $Z_{r-1}$ is the intersection scheme of $\ell$ with  $S_{r-1}$ and $x:=\ell \cap S_1$.   By generality of $\ell$,  we have  $x \not \in S_{r-1}$, so that $x \not \in Z_{r-1}$.
By \eqref{eq:sittinginext}, we have
\begin{equation}
  \label{eq:c2}
  [C_{r}]=c_2(\E_{r})=c_2(\E'_{r})= \left[C_{r-1} \cup \ell' \cup I\right], \; \; \mbox{with} \; \; I:= S_1 \cap S_{r-1},
\end{equation}
where
$C_{r-1}$ sits in 
\[ \xymatrix{
 0 \ar[r]   & \CC^{r-2}  \* \O_X \ar[r] & \E_{r-1} \ar[r] & \I_{C_{r-1}/X}(r-1) \ar[r] & 0. 
}
\]
We want to prove that none of the curves in the limit linear systems of $|\O_{S_{r}}(C_{r})|$ contains  the scheme  $Z_{r-1} \cup x$. To do so, consider the family $\pi: \mathcal{S} \to \DD$, whose general fiber is a smooth surface $S_{r} \in |\O_X(r)|$  and      whose   central fiber is $S_{r-1} \cup S_1$.  Then, by \cite[Chp. 2]{Ser} or \cite[\S 2]{Fr}, the singular locus  $\mathfrak Z$ of $\mathcal{S}$  is supported at a  divisor in the linear system 
$|\N_{I/S_1} \* \N_{I/S_{r-1}}|$ on $I$, where $\N_{I/S_j}$ denotes the normal bundle of $I$ in $S_j$ for $j=1, r-1$, which has degree $(r-1)^2d+(r-1)d=r(r-1)d$. One can make $\mathcal{S}$ smooth by making a small resolution at each of the singular points, cf., e.g., \cite[p.~647]{CLM}, in such a way that the general fiber is unaltered and the central  fibre  is replaced by $S_{r-1} \cup \widetilde{S}_1$, where $\widetilde{S}_1 \to S_1$ is  the blow--up along the points of $\mathfrak Z$ on $I$.  By abuse of notation, we still denote by $\ell'$ the strict transform  of $\ell'$   
on $\widetilde{S}_1$. We also still denote $S_{r-1} \cap \widetilde{S}_1$ by $I$.

A  limit linear system of $|\O_{S_{r}}(C_{r})|$ consists of the union of suitable linear systems on $S_{r-1}$ and $\widetilde{S}_1$ matching along $I$. Precisely,  by \eqref{eq:c2},   a limit linear system is of the form  
\[ \Lambda_k:=\Big\{(D_1,D_2) \in |\O_{S_{r-1}}(C_{r-1}+I-kI)| \x |\O_{\widetilde{S}_1}(\ell'+kI)| \; : \;  D_1 \cap I=D_2 \cap I\Big\}, \]
 and  we have $k \geqslant 0$ since $|\O_{\widetilde{S}_1}(\ell'-I)|$ is empty.  
By induction, no member of $|\O_{S_{r-1}}(C_{r-1})|$ contains $Z_{r-1}$, whence also no member of $|\O_{S_{r-1}}(C_{r-1}+I-kI)|$, for any $k >0$, contains $Z_{r-1}$.  Hence,  for $k>0$  no curve in  $\Lambda_k$ contains  $Z_{r-1}\cup x$. Let us see that also no curve in $\Lambda_0$ contains $Z_{r-1}\cup x$. In fact, as  $\ell$ and $\ell'$ are general lines on $X$, they do not intersect, whence $x \not \in \ell'$. Since for $k=0$ we have $ |\O_{\widetilde{S}_1}(\ell'+kI)|=|\O_{\widetilde{S}_1}(\ell')|=\{\ell'\}$, we see that no curve in $\Lambda_0$ contains $Z_{r-1}\cup x$.  Thus, \eqref{eq:grant'} follows, proving \eqref{eq:grant}.

 (ii)  Consider the short exact sequence
\[\xymatrix{
    0 \ar[r]   & \E_{r}^*(-1) \*\I_{\ell/X} \ar[r] & \E_{r}^*(-1)  \ar[r] &
\E_{r}^*(-1)|_{\ell} \ar[r] & 0 
}
\]
Since $\E_{r}$ is Ulrich, we have $h^i(\E_{r}^*(-1))=h^{3-i}(\E_{r}(-1))=0$ for $i \in \{0,1,2,3\}$. 
Moreover, by (i) we have
$h^1( \E_{r}^*(-1)|_{\ell})=h^1(\O_{\ell}(-2)^{\+ r})=r$,
and the result follows.
\end{proof}

Since
\begin{equation} \label{eq:Ext0}
  \Ext_X^1(\I_{\ell/X}(1),\E_{r}) \cong \Ext_X^1(\I_{\ell/X}(-1) \* \E_{r}^*,\omega_X) \cong H^2(\E_{r}^*(-1) \*\I_{\ell/X})^* \cong \CC^{r}
  \end{equation}
by Lemma \ref{lemma:Egen-1}(ii), there is  an $(r-1)$-dimensional family of  nonsplit  extensions
\begin{equation}\label{eq:Ext2}
 \xymatrix{
 0 \ar[r]   & \E_{r} \ar[r] & \F \ar[r] & \I_{\ell/X}(1) \ar[r] & 0. 
}
\end{equation}

 In the next three lemmas (\ref{lemma:Flocfree}--\ref{lemma:uniquedest}) we list a number of  properties of any sheaf $\F$ obtained in this way.

\begin{lemma} \label{lemma:Flocfree}  Any sheaf $\F$  as in \eqref {eq:Ext2}
is locally free. 
\end{lemma}

\begin{proof}
  From the sequence
  \begin{equation} \label{eq:idealell}
    \xymatrix{
   0 \ar[r]   & \I_{\ell/X} \ar[r] & \O_X \ar[r] & \O_{\ell} \ar[r] & 0
 }
 \end{equation}
twisted by $\O_X(1)$ we find that
\[\Shext^i_{\O_X}(\I_{\ell/X}(1),\O_X) \cong \Shext^{i+1}_{\O_X}(\O_{\ell}(1),\O_X)\cong 0 \; \; \mbox{for} \; \; i \geqslant 2,
  \]
 (as $\codim(\ell,X)=2$) and
  \[\Shext^1_{\O_X}(\I_{\ell/X}(1),\O_X) \cong \Shext^{2}_{\O_X}(\O_{\ell}(1),\O_X) \cong \Shext^{2}_{\O_X}(\O_{\ell},\omega_X)\otimes \mathcal O_X (1) \cong \omega_{\ell}(1) \cong \O_{\PP^1}(-1).
 \]
By \eqref{eq:Ext2} and the fact that $\E_{r}$ is locally free, we therefore get
 \[ \Shext^i_{\O_X}(\F,\O_X) \cong  \Shext^i_{\O_X}(\I_{\ell/X}(1),\O_X) =0, \; \; \mbox{for} \; \; i \geqslant 2\]
and
\begin{equation}\label{eq:Ext2dual}
 \xymatrix{
   0 \ar[r]   & \O_X(-1) \ar[r] & \F^* \ar[r] & \E_{r}^* \ar[r]^{\hspace{-0.5cm}\alpha} &
\O_{\PP^1}(-1) \ar[r] &  \Shext^1_{\O_X}(\F,\O_X)\ar[r] & 0. 
}
\end{equation}
To prove that $\F$ is locally free, we need to show that $\alpha$ is surjective. We argue by contradiction.

If $\alpha=0$, then dualizing the left part of \eqref{eq:Ext2dual} we obtain
\begin{equation}\label{eq:Ext2'}
 \xymatrix{
 0 \ar[r]   & \E_{r} \ar[r] & \F^{**} \ar[r] & \O_{X}(1) \ar[r] & 0. 
}
\end{equation}
 Ulrichness of $\E_r$ yields 
\[ \Ext_X^1(\O_X(1),\E_{r}) \cong \Ext_X^1(\O_X,\E_{r}(-1)) \cong H^1(\E_{r}(-1))=0,\]
 whence  \eqref{eq:Ext2'} splits. Combining
\eqref{eq:Ext2} and \eqref{eq:Ext2'} with the natural map $h:\F \to \F^{**}$, we get
\[
  \xymatrix{0 \ar[r]   & \E_{r} \ar[r]^i \ar@{=}[d] & \F \ar[r] \ar[d]^{h}& \I_{\ell/X}(1) \ar[r] \ar[d] & 0 \\
  0 \ar[r]   & \E_{r} \ar[r]^j & \F^{**} \ar[r] \ar@/^{1.0pc}/[l]^{p} & \O_{X}(1) \ar[r] & 0,}
\]
where $p: \F^{**} \to \E_{r}$ is the projection map induced by the splitting 
of \eqref{eq:Ext2'}. Since $p \circ j=\id_{\E_{r}}$, we see that $p \circ h:\F \to \E_{r}$ induces a splitting of \eqref{eq:Ext2} as well, a contradiction.

If $\alpha \neq 0$ and $\alpha$ is not surjective, then $\im (\alpha) \cong \O_{\PP^1}(-1-a)$ for some  $a >0$. 
  By  Lemma \ref{lemma:Egen-1}(i)  we would obtain a nonzero morphism
  $\O_{\PP^1}(-1) \to \O_{\PP^1}(-1-a)$ for $a>0$, which is impossible.
\end{proof}

\begin{lemma} \label{lemma:vanUl}  Let $\F$  be as in \eqref {eq:Ext2}. Then: 
    \begin{itemize}
  \item[(i)] $h^i(\F(-j))=0$ for all $i \geqslant 0$ and $j \in \{1,2\}$,
  \item[(ii)] $h^i(\F(-3))=
    \begin{cases}
      0, & i \in \{0,1\},\\
      1, & i \in \{2,3\},
    \end{cases}
    $
  \item[(iii)] $h^3(\F \* \F^*)=0$,
\item[(iv)] $\chi(\F \* \F^*)=-(r+1)^2$.
  \end{itemize}
\end{lemma}

\begin{proof}
  Items (i)-(ii) follow by   straightforward computations  using the sequences \eqref{eq:Ext2} and \eqref{eq:idealell}, suitably twisted.

By   tensoring  \eqref{eq:Ext2} by $\F^*$ and taking cohomology we find
\[ h^3(\F \* \F^*) \leqslant h^3(\E_{r} \* \F^*)+h^3(\F^* \* \I_{\ell/X}(1)).\]
By (i) and Lemma \ref{lemma:vaniz}(i) we get $h^3(\E_{r} \* \F^*)=0$. Moreover,
\[h^3(\F^* \* \I_{\ell/X}(1)) \leqslant h^2( \F^*(1)|_{\ell})+h^3(\F^*(1))= h^0(\F(-3))=0,\]
again by (i). This proves (iii).

Finally, by \eqref{eq:Ext2} we have:
  \begin{eqnarray*}
    c_1(\F) & = & c_1(\E_{r})+c_1(\O_X(1))=c_1(\O_X(r+1)), \\
    c_2(\F) \cdot \O_X(1) & = & \Big(c_2(\E_{r})+c_2(\I_{\ell/X}(1))+c_1(\E_{r})\cdot c_1(\I_{\ell/X}(1))\Big)\cdot \O_X(1) \\
                     & = & \Big([C_{r}]+[\ell]+[\O_X(r) \cdot \O_X(1)]\Big) \cdot \O_X(1) = \deg(C_{r})+\deg(\ell)+r d \\
    & = &   \frac{r}{2}\left(r d-d+2\right)+1+r d = \frac{r+1}{2}\left((r+1) d-d+2\right)
  \end{eqnarray*}
(where we have used \eqref{eq:fanoindex2-2}); thus (iv) follows from Lemma \ref{lemma:vaniz}(ii). 
\end{proof}

\begin{lemma} \label{lemma:uniquedest}  Let $\F$  be as in \eqref {eq:Ext2} and
let  $\G$ be a destabilizing subsheaf of $\F$. Then $\G^{*} \cong  \E_{r}^*$. 
  In particular, $\mu(\G) = \mu(\F)$ whence $\F$ is slope--semistable. 
\end{lemma}

\begin{proof}
  We note that \eqref{eq:Ext2} yields
    $\mu(\F)=\frac{\O_X(r+1)\cdot \O_X(1)^2}{r+1}=d$.  
Assume that $\G$ is a destabilizing subsheaf of $\F$, that is  $0<\rk(\G) \leqslant \rk(\F)-1=r$ and $\mu(\G) \geqslant d$. Define
  \[ \Q:=\im\{\G \subset \F \longrightarrow \I_{\ell/X}(1)\} \; \; \mbox{and} \; \; \K:=\ker\{\G \longrightarrow \Q\}.\]
  Then we may put \eqref{eq:Ext2} into a commutative diagram with exact rows and columns:
  \[
    \xymatrix{ & 0 \ar[d] & 0 \ar[d] & 0 \ar[d] & \\
  0 \ar[r]   & \K \ar[d] \ar[r] & \G \ar[r] \ar[d] & \Q \ar[r] \ar[d] & 0 \\    
  0 \ar[r]   & \E_{r} \ar[r] \ar[d] & \F \ar[r] \ar[d] & \I_{\ell/X}(1) \ar[d] \ar[r] & 0 \\
  0 \ar[r]   & \K' \ar[r] \ar[d] & \F/\G \ar[r] \ar[d] & \Q' \ar[r] \ar[d] & 0 \\
  & 0  & 0  & 0  &
}
\]
defining $\K'$ and $\Q'$. We have $\rk(\Q) \leqslant 1$.

Assume that $\rk(\Q)=0$. Since $\I_{\ell/X}(1)$ is torsion-free, we must have $\Q =0$, whence $\K \cong \G$. Since $\mu(\K) =\mu(\G) \geqslant d=\mu(\E_{r})$ and $\E_{r}$ is slope--stable (being a general member of $\UU_r$), we must have
$\rk(\K)=\rk(\E_{r})=r$. It follows that $\rk(\K')=0$.  Since
\[
  c_1(\K) =  c_1(\E_{r})-c_1(\K')=[\O_X(r)]-D',\]
where $D'$ is an effective divisor supported on the codimension one locus of the support of $\K'$, we have 
\[ d \leqslant \mu(\K)=\frac{\left(\O_X(r)-D'\right) \cdot \O_X(1)^2}{r}= d-\frac{D' \cdot \O_X(1)^2}{r}.\]
  Hence $D'=0$   (and $\mu (\K) = d$),   which means that $\K'$ is supported in codimension at least two.
  Thus, $\Shext_{\O_X}^i(\K',\O_X)=0$ for $i \leqslant 1$, and it follows that $\K^* \cong \E_{r}^*$ and   $\mu(\G) = \mu(\F)$,   as desired.

Assume that $\rk(\Q)=1$. Then $\rk(\K)=\rk(\G)-1 \leqslant r-1<r=\rk(\E_{r})$
and $\rk(\Q')=0$; in particular $\Q'$ is a torsion sheaf. Since
\[
  c_1(\K) =  c_1(\G)-c_1(\Q)=c_1(\G)-c_1(\I_{\ell/X}(1))+c_1(\Q')=c_1(\G)-c_1(\O_X(1))+D,
\]
where $D$ is an effective divisor supported on the codimension one locus of the support of $\Q'$, we have
\begin{eqnarray*}
  \mu(\K) & = & \frac{\Big(c_1(\G)-c_1(\O_X(1))+D\Big) \cdot \O_X(1)^2}{\rk(\K)} \geqslant \frac{c_1(\G)\cdot \O_X(1)^2-\O_X(1)^3}{\rk(\K)}
  \\
          & = & \frac{\mu(\G)\rk(\G)-d}{\rk(\K)} \geqslant
                \frac{d\rk(\G)-d}{\rk(\K)} =\frac{d(\rk(\K)+1)-d}{\rk(\K)}=d.
\end{eqnarray*}
This contradicts  the  slope-stability of $\E_{r}$.
\end{proof}

We denote by $\UU_{r+1}^{\rm{ext}}$ the family of vector bundles on $X$ 
consisting of all extensions $\F$ of the form \eqref{eq:Ext2},   for a general (stable) $\E_{r}$ in $\UU_r$   
and $\ell$ a general line in (a component of) $\H_1(X)$.   
Note that all the elements are semistable by the previous lemma. We denote by 
$\UU_{r+1}$ the component of the moduli space of semistable bundles on $X$  containing $\UU_{r+1}^{\rm{ext}}$.   
To finish the proof of Proposition \ref{prop:indrho}, we need to prove that the general member in $\UU_{r+1}$ is Ulrich and slope--stable. This will be done in Lemma \ref{lemma:genUst} below,  after  an intermediate result:

\begin{lemma} \label{lemma:propcont}  One has
  $\dim(\UU_{r+1}^{\rm {ext}})<\dim(\UU_{r+1})$.
\end{lemma}

\begin{proof}
  By Lemma \ref{lemma:vanUl}(iii)-(iv)  and standard deformation theory  we have
  \begin{eqnarray*}
    \dim(\UU_{r+1}) & \geqslant & h^1(\F \* \F^*)-h^2(\F \* \F^*) =-\chi(\F \* \F^*)+h^0(\F \* \F^*) \\
    & \geqslant & (r+1)^2+1 =  r^2+2r+2.
                            \end{eqnarray*}
 On the other hand,  using \eqref{eq:Ext0}, we have
  \begin{eqnarray*} \dim(\UU_{r+1}^{\rm {ext}}) & \leqslant & \dim(\UU_r)+\dim \PP(\Ext_X^1(\I_{\ell/X}(1),\E_{r}))+\dim(\H_1(X))\\
    & = & (r^2+1)+(r-1)+2=r^2+r+2.
    \end{eqnarray*}  
The result follows. 
\end{proof}

\begin{lemma} \label{lemma:genUst}
  The general member in $\UU_{r+1}$ is Ulrich and slope--stable.
\end{lemma}

\begin{proof}
  Let $\E$ be a general member of $\UU_{r+1}$. By Lemma \ref{lemma:vanUl}(i)-(ii) and semicontinuity we have $h^i(\E(-j))=0$ for all $i \geqslant 0$ and $j \in \{1,2\}$, $h^i(\E(-3))=0$ for $i \in \{0,1\}$ and $\chi(\E(-3))=0$. To prove that $\E$ is Ulrich, we therefore have left to prove that $h^3(\E(-3))=0$, equivalently $h^0(\E^*(1))=0$.

  Assume therefore, to get a contradiction, that $h^0(\E^*(1))>0$ for a  general $\E$  in $\UU_{r+1}$, and consider a  nonzero  section $s \in H^0(\E^*(1))$. Setting $\Q:=\coker(s)$, we have
  \begin{equation}
    \label{eq:Estab}
    \xymatrix{0 \ar[r] & \O_X \ar[r]^{s} & \E^*(1)\ar[r] & \Q \ar[r] & 0.}
  \end{equation}
 We may assume that this specializes to a nonzero section  $s_0 \in H^0(\F^*(1))$ for all $\F$ in $\UU_{r+1}^{\rm {ext}}$ and  we have  an exact sequence
  \[
\xymatrix{0 \ar[r] & \O_X \ar[r]^{s_0} & \F^*(1)\ar[r] & \Q_0 \ar[r] & 0,}
    \]
    where $\Q_0:=\coker(s_0)$. Dualizing and twisting by $\O_X(1)$, we obtain
 \[
\xymatrix{0 \ar[r] & \Q_0^*(1) \ar[r] & \F \ar[r]^{\hspace{-0.3cm}s_0^*} & \O_X(1),}
    \]   
    and  one has  $\im(s_0^*)=\I_{Z/X}(1)$ for some (possibly empty) subscheme $Z \subset X$. Since
    \begin{eqnarray*} c_1(\Q_0^*(1)) & = & c_1(\F)-c_1(\I_{Z/X}(1))=c_1(\O_X(r+1))-c_1(\O_X(1))+c_1(\O_Z(1))\\
      & = & c_1(\O_X(r))+Z',
\end{eqnarray*}
where $Z'$ is an effective divisor  supported on the codimension-one locus of $Z$, we have
\begin{eqnarray*}
  \mu(\Q_0^*(1)) & = & \frac{\Big(c_1(\O_X(r))+Z'\Big) \cdot \O_X(1)^2}{\rk(\Q_0^*(1))} \geqslant \frac{c_1(\O_X(r))\cdot \O_X(1)^2}{r}=\frac{r d}{r}=d.
\end{eqnarray*}
Lemma \ref{lemma:uniquedest} therefore yields  $\left(\Q_0^*(1)\right)^* \cong \E_{r}^*$, whence $\Q_0^*(1) \cong \E_{r}$, as $\Q_0^*(1)$ is reflexive (by \cite[Prop. 1.1]{har}).  It follows that $\Q^*(1)$ is a deformation of $\E_{r}$, that is, $[\Q^*(1)] \in \UU_r$. The dual of \eqref{eq:Estab} twisted by $\O_X(1)$  gives 
 \[
\xymatrix{0 \ar[r] & \Q^*(1) \ar[r] & \E \ar[r]^{\hspace{-0.3cm}s^*} & \O_X(1),}
    \]
and for reasons of Chern classes we must have $\im(s^*)\cong \I_{\ell/X}(1)$ for a line $\ell \subset X$. Therefore $[\E] \in \UU_{r+1}^{\rm {ext}}$, contradicting Lemma \ref{lemma:propcont}.
We have  thus   proved that $\E$ is Ulrich.

If the general member of $\UU_{r+1}$ were not slope--stable, we could find a one-parameter  flat  family of bundles $\{\E^{(t)}\}$ over the disc $\DD$ such that $\E^{(t)}$ is a general member of $\UU_{r+1}$ for $t \neq 0$ and $\E^{(0)}$ lies in $\UU_{r+1}^{\rm {ext}}$, and such that we have, for $t \neq 0$,  a destabilizing sequence
\begin{equation} \label{eq:destat} \xymatrix{
    0 \ar[r] & \G^{(t)} \ar[r] & \E^{(t)} \ar[r] & \F^{(t)} \ar[r] & 0,}
  \end{equation}
   which  we can take to be saturated, that is, such that $\F^{(t)}$ is torsion free, whence so that $\F^{(t)}$ and $\G^{(t)}$ are (Ulrich) vector bundles  (see \cite[Thm. 2.9]{ch} or \cite[(3.2)]{be2}).

  The limit of $\PP(\F^{(t)}) \subset \PP(\E^{(t)})$ defines a subvariety of $\PP(\E^{(0)})$ of the same dimension as $\PP(\F^{(t)})$, whence a coherent sheaf $\F^{(0)}$ of rank $\rk(\F^{(t)})$ with a surjection $\E^{(0)} \twoheadrightarrow \F^{(0)}$. Denoting by $\G^{(0)}$ its kernel, we have
$\rk(\G^{(0)})=\rk(\G^{(t)})$ and $c_1(\G^{(0)})=c_1(\G^{(t)})$. Hence, \eqref{eq:destat} specializes to a destabilizing sequence for $t=0$. 
Lemma \ref{lemma:uniquedest} yields  that ${\G^{(0)}}^*$ is the dual of a member of $\UU_r$. It follows that
${\G^{(t)}}^*$ is a deformation of the dual of a member of $\UU_r$,  whence $\G^{(t)}$  is a deformation of a member of $\UU_r$, as both are locally free.
Therefore, $[\G^{(t)}] \in \UU_r$ and we obtain the same contradiction to Lemma \ref{lemma:propcont} as above.
\end{proof}

The proof of Proposition \ref{prop:indrho} is now complete,  whence  also the proof of Theorem \ref{thm:Fanoind2}.

  \subsection{Conclusion of the proof of Theorem \ref{thm:veronese}(ii)}\label{sec:ver}

Let $X \subset \PP^{9}$  be the $2$--Veronese embedding of $\PP^3$.    We imitate the proof of the existence of Ulrich bundles of rank $r\geq 2$ on a general cubic threefold in   \cite[Thm.\;5.7]{ch}.  The strategy is as follows.   We define $\UU_2$ to be any component 
of the moduli space of slope--stable rank--$2$ Ulrich bundles of determinant $\O_X(2)$, which is nonempty by Propositions \ref{prop:beauville} and \ref{prop:beauville-stable}. 
  Then we proceed to construct inductively   (irreducible components of) moduli spaces $\UU_{2h}$   of rank--$2h$ Ulrich bundles, 
for all integers $h\geq 2$, in the following way:    assume that we have constructed, for an integer $h\geq 1$, an irreducible moduli space $\UU_{2h}$ 
parametrizing rank--2h Ulrich bundles of determinant $\O_X(2h)$ whose general member is slope--stable. In particular, for any $[\E_{2h}] \in \UU_{2h}$, we have that
\begin{equation}\label{eq:c2veronese}
c_2(\E_{2h}) \cdot \O_X(1) = 2h (8h-3)
\end{equation} by \eqref{eq:fano2} in Theorem \ref{thm:fano} with $\alpha =2$   (cf. Remark \ref{rem:c2}).  

\begin{lemma}\label{lem:kjy} For general $[\E_2]\in \UU_2$ and $[\E'_{2h}]\in \UU_{2h}$, such that $\E_2 \not \cong \E'_{2}$ when $h=1$,  
one has $\dim (\Ext^1(\E'_{2h}, \E_2))=4h$. 
\end{lemma}

\begin{proof} By Lemma \ref{lemma:vaniz}(i) we have $h^2(\E_2\otimes (\E'_{2h})^*)=h^3(\E_2\otimes (\E'_{2h})^*)=0$. 
Moreover $h^0(\E_2\otimes (\E'_{2h})^*)=0$, because $\E_2$ and $\E'_{2h}$ are   slope--stable bundles 
of the same slope, namely $8$, which are not isomorphic. Hence
$$
\dim (\Ext^1(\E'_{2h}, \E_2))= h^1(\E_2\otimes (\E'_{2h})^*)=-\chi(\E_2\otimes (\E'_{2h})^*) = 4h,
$$ by Lemma \ref{lemma:vaniz}(ii), since conditions \eqref{eq:Milan} hold true by \eqref{eq:c2veronese}. 
\end{proof}

By the previous lemma, there exist nonsplit extensions of the form 
\begin{equation}\label{eq:opla}
0\longrightarrow \E_2 \longrightarrow \F_{2h+2}\longrightarrow \E'_{2h}\longrightarrow 0, 
\end{equation} for general $[\E_2] \in \UU_2$ and $[\E'_{2h}] \in \UU_{2h}$. We have that 
$\rk( \F_{2h+2}) = 2h+2$, $c_1(\F_{2h+2})=[\O_X(2h+2)]$ and $\F_{2h+2}$ is easily seen to be Ulrich. 
We denote by $\UU_{2h+2}^{\rm ext}$ the family of such nonsplit extentions and we define 
$\UU_{2h+2}$ to be the component of the moduli space of Ulrich bundles  containing $\UU_{2h+2}^{\rm ext}$. 

To finish the proof of Theorem \ref {thm:veronese}(ii), we have left to prove that the general element of $\UU_{2h+2}$ 
is slope--stable. To this end, we need the following two lemmas. 

\begin{lemma}\label{lem:dimextvero} One has $\dim(\UU_{2h+2}^{\rm ext}) < \dim (\UU_{2h+2})$.
\end{lemma}  

\begin{proof} One checks that  in \eqref{eq:opla} one has $\mu(\E_2)=\mu(\E'_{2h})=8$. Thus, by \cite [Lemma\,4.2]{ch}, the bundle $\F_{2h+2}$ is simple, 
  whence $\dim (\UU_{2h+2}) = (2h+2)^2+1 = 4h^2 + 8 h +5$ by Proposition \ref{prop:dimmod}. On the other hand, by Lemma \ref {lem:kjy},   we have
\begin{eqnarray*}
\dim(\UU_{2h+2}^{\rm ext}) & \leq & \dim(\UU_2)+\dim(\UU_{2h})+\dim (\PP({\rm Ext}^1(\E'_{2h}, \E_2)))\\
 & = & 5+(4h^2+1)+(4h-1) = 4h^2 + 4 h + 5.
\end{eqnarray*}
\end{proof}

  \begin{lemma} \label{lemma:uniquedest2}  Let $\F_{2h+2}$  be as in \eqref {eq:opla} and
let  $\G$ be a destabilizing subsheaf of $\F_{2h+2}$. Then $\G^{*} \cong  \E_{2}^*$. 
\end{lemma}
\begin{proof} This is identical to the proof of Lemma \ref{lemma:uniquedest} and left to the reader. 
\end{proof}

To prove that the general element of $\UU_{2h+2}$ is slope--stable one applies almost verbatim the 
argument of the last part of the proof of Lemma \ref{lemma:genUst}: indeed, arguing by contradiction, if it were not slope--stable 
then one applies Lemma \ref{lemma:uniquedest2} to prove that it lies in $\UU_{2h+2}^{\rm ext}$, which contradicts Lemma \ref{lem:dimextvero}. 
This completes the proof of Theorem \ref{thm:veronese}.

\section{Final remarks and speculations}\label{S:moduli} 

Let  $ X\subset \PP^{d+1}$ be a smooth 	  Fano threefold   of degree $d$ and   index two with $\omega_X \cong \O_X(-2)$. 
Theorem \ref{thm:Fanoind2} proves that for  every  $r \geqslant 2$, there exists 
a smooth, $\left(r^2+1\right)$--dimensional moduli space  $\UU_{X,r}$, parametrizing a  family of Ulrich bundles of rank $r$ and determinant $\O_X(r)$ on $X$.  For any $\E\in \UU_{X,r}$, one has $h^0(\E)=r d$ (cf., e.g.,  \cite [(3.1)]{be2}). By  Theorems \ref{thm:codim2} and \ref{thm:fano}, given a general subspace $V$ of dimension $r-1$ in $H^0(\E)$, the inclusion $V\otimes \O_X\longrightarrow \E$ drops rank along a smooth curve $C$  of genus
\color{black} \[ g_{r,d}=\frac{1}{3}rd(r^2-3r+2)+(r-1)^2 \]\color{black}
(cf. \eqref{eq:fano1} and \eqref{eq:fano3}).   

 Consider the moduli space $\ff_{2,d}$ of isomorphisms classes of smooth Fano threefolds of index two and degree $d$  and the universal moduli space $\UU_{r,d}$ of Ulrich bundles of rank $r$ over  $\ff_{2,d}$:   a   point in $\UU_{r,d}$ is a pair $([X],[\E])$, where $[X] \in \ff_{2,d}$  and $[\E]\in \UU_{X,r}$.   Let $\mathbb G$ be the  Grassmannian bundle over $\UU_{r,d}$, whose fibre 
over $([X],[\E])$ is $\Grass(r-1,H^0(X,\E))$. By the above considerations, we have a  rational map
\[
 \varphi_{r,d}: \mathbb G \dasharrow  \M_{g_{r,d}},
\]
 where  $\M_g$  is the moduli space of smooth projective curves of genus 
  $g$. Its image  is clearly uniruled, and, as such, it is an interesting sublocus of {$\M_{g_{r,d}}$.  In particular, it is a natural question to understand  for which $d$ and $r$  the map $\varphi_{r,d}$ is dominant. 

Let us focus on low values of $r$. For  $r=2$ one has $g_{2,d}=1$ and it is easy to deduce from the proof of Proposition \,\ref{prop:beauville} that $\varphi_{2,d}$ is dominant for all $d$. 

The next step is $r =3$. In this case
 $g_{3,d} = 2d +4$, and, as we have seen in Example \ref{ex:fanoindex2}, through the embedding in $X$ the curve $C$ admits a linearly and quadratically normal embedding as a curve of degree $3d+3$ in $\PP^{d+1}$ with injective Petri map $\mu_{0,\O_C(1)}$ \linebreak (cf. \eqref{eq:petri}).
(We also note that  the linear system 
 $|\O_{C}(1)|$  has Brill--Noether number $\rho(2d+4, d+1, 3d+3) = 0$.) This 
suggests that the map $\varphi_{3,d}$ might also be dominant. 

A result in this direction is the following:

\begin{proposition}\label{prop:su} The map $\varphi_{3,d}$ is dominant for $3\leq d\leq 4$.
\end{proposition} 

\color{black}
\begin{proof}[Idea of the proof]  We  give an idea in the case $d=3$, where the target space is $\M_{10}$. The proof in the case $d=4$, where the target space is $\M_{12}$, is similar.

Since $\rho(10, 4, 12) = 0$, there are linearly normal curves with general moduli in $\mathbb P^4$ of degree 12 and genus 10. It suffices to prove that such a general curve lies on a smooth cubic hypersurface; indeed the remaining properties as in Example \ref{ex:fanoindex2}(ii) are automatically satisfied since the curve has general moduli. Actually one proves that the general such curve lies on a general cubic hypersurface.

To do so, let $\mathcal H$ be the component of the Hilbert scheme containing the above curves with general moduli. An easy count of parameters shows that $\dim(\mathcal H)=51$. Consider  the incidence variety
$$
I=\{(C,F)\in \mathcal H\times |\mathcal O_{\mathbb P^4}(3)|: C\subset F\}
$$
with the two projections $p_1: I\to \mathcal H$ and $p_2: I\to |\mathcal O_{\mathbb P^4}(3)|$. The map $p_1$ is dominant and the general fibre has dimension 7 (because the general curve $C\in \mathcal H$ is of maximal rank). So $\dim(I)=58$. One proves that $p_2$ is dominant by proving that the general fibre of $p_2$ has dimension  24 (recall that $\dim(|\mathcal O_{\mathbb P^4}(3)|)=34$). It suffices to show that there is one $F\in |\mathcal O_{\mathbb P^4}(3)|$ such that $p_2^{-1}(F)$ has dimension 24.  One proves this by degeneration, showing that this is the case if $F$ is the general union of a quadric and a hyperplane, which is the technical and lengthy part of the proof that we skip.  \end{proof}

\color{black}

As a consequence, one has that
 $\M_{10}$  and 
 $\M_{12}$  are uniruled, which is no  big  news,  because  we  know since a long time that  $\M_{10}$ and $\M_{12}$ are   
unirational  (cf. \cite{arse,Ser0}). 

The proof of Proposition \ref {prop:su}, based on delicate counts of parameters, is rather lenghty and we do not dwell on it here.

The question remains, how does $\varphi_{3,d}$ behave for $5\leq d\leq 7$? We suspect that in these cases $\varphi_{3,d}$ is dominant onto   a codimension 3 closed subset of  $\M_{2d+4}$.  Let us briefly explain the reason  for   this expectation. 
First of all   note  that the moduli spaces  $\ff_{2,d}$ consist of one isolated point for $5\leq d\leq 7$ (see  \cite {IP}). Hence  
$$
\dim(\mathbb G)=10+2(3d-2)=6d+6.
$$
Since $\dim(\M_{2d+4})=6d+9$, the map $\varphi_{3,d}$ cannot be dominant. We conjecture that it is generically finite onto its image, which would prove that $\varphi_{3,d}$  dominates a  codimension 3 closed subset of  $\M_{2d+4}$.

\end{document}